\documentclass[12pt]{article}
\usepackage{amsmath,amssymb,amsthm,comment}

\newtheorem{theorem}{Theorem}[section]
\newtheorem{lemma}[theorem]{Lemma}
\newtheorem{corollary}[theorem]{Corollary}

\def\barr{\begin{array}}
\def\earr{\end{array}}

\title{On the supersolvability of a finite group by the sum of subgroup orders}
\author{Marius T\u arn\u auceanu}
\date{February 14, 2021}

\begin{document}

\maketitle

\begin{abstract}
Let $G$ be a finite group and $\sigma_1(G)=\frac{1}{|G|}\sum_{H\leq G}\,|H|$. In this paper, we prove that if $\sigma_1(G)<2+\frac{11}{|G|}$\,, then $G$ is supersolvable. In particular, some new characterizations of the well-known groups $\mathbb{Z}_2\times\mathbb{Z}_4$ and $A_4$ are obtained. We also show that $\sigma_1(G)<c$ does not imply the supersolvability of $G$ for no constant $c\in(2,\infty)$.
\end{abstract}

{\small
\noindent
{\bf MSC2000\,:} Primary 20D60; Secondary 20D10, 20F16, 20F17.

\noindent
{\bf Key words\,:} subgroup orders, supersolvable groups.}

\section{Introduction}

Given a finite group $G$, we consider the function
$$\sigma_1(G)=\frac{1}{|G|}\sum_{H\leq G}\,|H|$$studied in our previous papers \cite{7,10,11}. Recall some basic properties of $\sigma_1$:
\begin{itemize}
\item[-] if $G$ is cyclic of order $n$ and $\sigma(n)$ denotes the sum of all divisors of $n$, then $\sigma_1(G)=\frac{\sigma(n)}{n}$\,;
\item[-] $\sigma_1$ is multiplicative, i.e. if $G_i$, $i=1,2,\dots,m$, are finite groups of coprime orders, then $\sigma_1(\prod_{i=1}^m G_i)=\prod_{i=1}^m \sigma_1(G_i)$;
\item[-] $\sigma_1(G)\geq\sigma_1(G/H)+\frac{1}{(G:H)}\,(\sigma_1(H)-1)\geq\sigma_1(G/H)$, for all $H\unlhd G$.
\end{itemize}

The function $\sigma_1$ was used to provide some criteria for a group to belong to one of the well-known classes of groups. For instance, if $\sigma_1(G)\leq 2$ then $G$ is cyclic of deficient or perfect order by \cite{7}, while if $\sigma_1(G)\leq\frac{117}{20}$\, then $G$ is solvable by \cite{4,11}. Also, we note that there is no constant $c\in(2,\infty)$ such that $\sigma_1(G)<c$ implies the nilpotency of $G$, but this can be obtained from the condition $\sigma_1(G)<2+\frac{4}{|G|}$ (see \cite{10}).

In the current paper, we will first determine the structure of finite groups $G$ satisfying $\sigma_1(G)<3$. This will allow us to prove that if $\sigma_1(G)<2+\frac{11}{|G|}$\, then $G$ is supersolvable. Moreover, $A_4$ is the unique finite non-nilpotent group $G$ such that $\sigma_1(G)=2+\frac{11}{|G|}$\,. Then we will show that $\sigma_1(G)<c$ does not imply the supersolvability of $G$ for no constant $c\in(2,\infty)$.
\smallskip

For the proof of our results, we need the following two theorems (see Theorem 1 of \cite{6} and Theorem 1 of \cite{8}).

\begin{theorem}
Suppose every non-normal maximal subgroup of a finite group $G$ has the same order. Then $G$ is solvable and the nilpotent length of $G$ is at most two.
\end{theorem}

\begin{theorem}
A finite group $G$ contains a cyclic maximal subgroup if and only if it is of one of the following types:
\begin{itemize}
\item[{\rm a)}] $G=P\times K_1$, where $K_1$ is an arbitrary finite cyclic group and $P$ is a Sylow $p$-subgroup of $G$ containing a cyclic maximal subgroup.
\item[{\rm b)}] $G=(K_2\rtimes P)\times K_1$, where $K_1$ is an arbitrary finite cyclic group which is a Hall subgroup of $G$, $P$ is a Sylow $p$-subgroup of $G$ containing a cyclic maximal subgroup, $K_2\rtimes P$ is non-nilpotent, and the centralizer $C_P(K_2)$ is a maximal cyclic subgroup of $P$.
\item[{\rm c)}] $G=(P\rtimes K_2)\times K_1$, where $K_1$ is an arbitrary finite cyclic group which is a Hall subgroup of $G$ and $G_1=P\rtimes K_2$ is a non-nilpotent group satisfying the following conditions: $P$ is a SyIow $p$-subgroup of $G_1$, $C_P(K_2)\supseteq\Phi(P)$, and $C_P(K_2)$ is a cyclic invariant subgroup of $G_1$ such that $K_2C_P(K_2)/C_P(K_2)$ is a maximal subgroup of $G_1/C_P(K_2)$.
\end{itemize}
\end{theorem}

We also need two auxiliary results, taken from \cite{11} and \cite{7}, respectively.

\begin{lemma}
Let $G$ be a finite group and $[M]$ be a conjugacy class of non-normal maximal subgroups of $G$. Then
$$\sum_{H\in[M]}|H|=|G|.$$
\end{lemma}\newpage

\begin{lemma}
Let $G$ be a finite group. Then
\begin{equation}
\sum_{H\leq G,\, H={\rm cyclic}}\!\!\!|H|\geq |G|.\nonumber
\end{equation}
\end{lemma}

Note that similar problems for some other functions related to the structure of a finite group $G$, for example for the function $\psi(G)=\sum_{x\in G}o(x)$ (where $o(x)$ denotes the order of the element $x$), have been recently investigated by many authors (see e.g. \cite{1,2,3}).
\smallskip

Most of our notation is standard and will not be repeated here. Basic definitions and results on groups can be found in \cite{5}. For subgroup lattice concepts we refer the reader to \cite{9}.

\section{Main results}

We start with the following important lemma.

\begin{lemma}
    Let $G$ be a finite group with $\sigma_1(G)<3$. Then either $G$ is nilpotent or a group of type c) in Theorem 1.2.
\end{lemma}

\begin{proof}
Assume that $G$ is not nilpotent. If $G$ has at least two conjugacy classes of non-normal maximal subgroups $[M_1]$ and $[M_2]$, then Lemma 1.3 leads to
\begin{equation}
\sigma_1(G)>\frac{1}{|G|}\left(|G|+\sum_{i=1}^2\sum_{H\in[M_i]}|H|\right)=\frac{1}{|G|}\left(|G|+2|G|\right)=3,\nonumber
\end{equation}a contradiction. Thus $G$ possesses a unique conjugacy class of non-normal maximal subgroups, say $[M]$. Moreover, $M$ is cyclic. Indeed, if this is not true, then Lemmas 1.3 and 1.4 imply
\begin{equation}
\sigma_1(G)\geq\frac{1}{|G|}\left(|G|+\!\!\sum_{H\in[M]}|H|+\!\!\!\!\!\!\!\!\sum_{H\leq G,\, H={\rm cyclic}}\!\!\!\!\!|H|\right)\geq\frac{1}{|G|}\left(|G|+|G|+|G|\right)=3,\nonumber
\end{equation}contradicting again our hypothesis.

From Theorem 1.1 it follows that $G$ is solvable of nilpotent length two. More precisely, by taking a prime $p$ dividing $[G:M]$, we infer that a Sylow $p$-subgroup $P$ of $G$ must be normal and a $p$-complement of $G$ must be cyclic. This shows that $G$ is a group of type c) in Theorem 1.2, completing the proof.\qedhere
\end{proof}

We are now able to prove our first main result.

\begin{theorem}\label{th:C1}
    Let $G$ be a finite group. If $\sigma_1(G)<2+\frac{11}{|G|}$\,, then $G$ is supersolvable.
\end{theorem}

\begin{proof}
Since all groups of order $\leq 11$ are supersolvable, we may assume that $|G|\geq 12$. Then $\sigma_1(G)<3$ and therefore, by the above lemma, either $G$ is nilpotent or $G=(P\rtimes K_2)\times K_1$ with $P$, $K_1$ and $K_2$ as in Theorem 1.2, c). In the first case we are done, while in the second case we may assume that $G=G_1=P\rtimes K_2$ because
\begin{equation}
\sigma_1(G_1)=\sigma_1(G/K_1)\leq\sigma_1(G)<2+\frac{11}{|G|}\leq 2+\frac{11}{|G_1|}\,.
\end{equation}Denote by $[M]$ the unique conjugacy class of non-normal maximal subgroups of $G$. Let $|P|=p^r$ and suppose that $P$ is not cyclic. Then $r\geq 2$ and $P$ contains at least $p+1$ subgroups of order $p^{r-1}$. It follows that
\begin{equation}
\barr{lcl}
\sigma_1(G)\!\!\!\!&\geq&\!\!\!\!\displaystyle\frac{1}{|G|}\left(|G|+\!\!\displaystyle\sum_{H\in [M]}|H|+\!\displaystyle\sum_{H\leq P}|H|\right)\!\!\geq\!2+\displaystyle\frac{1+p+\cdots+(p+1)p^{r-1}\!+p^r}{|G|}\vspace{1,5mm}\\
&\geq& \!\!\!\!2+\displaystyle\frac{1+(p+1)p+p^2}{|G|}=2+\displaystyle\frac{1+p+2p^2}{|G|}\geq 2+\displaystyle\frac{11}{|G|}\,,\earr\nonumber
\end{equation}contradicting (1). Thus $P$ is cyclic and consequently all Sylow subgroups of $G$ are cyclic, implying that $G$ is supersolvable. This completes the proof.\qedhere
\end{proof}

Next we will determine finite groups $G$ such that $\sigma_1(G)=2+\frac{11}{|G|}$\,. First of all, we will focus on $p$-groups. We observe that for all primes $p$ and all positive integers $n$ we have
\begin{equation}
\sigma_1(\mathbb{Z}_{p^n})=\displaystyle\frac{\sigma(p^n)}{p^n}=\displaystyle\frac{1+p+\cdots+p^n}{p^n}=1+\displaystyle\frac{1}{p-1}-\displaystyle\frac{1}{p^n(p-1)}<2.\nonumber
\end{equation}

The following lemma deals with the case of non-cyclic $p$-groups.

\begin{lemma}
Let $G$ be a non-cyclic $p$-group of order $p^n$. Then:
\begin{itemize}
\item[{\rm a)}] $\sigma_1(G)<2+\frac{11}{|G|}$\, if and only if either $n=3$, $p=2$ and $G=Q_8$, or $n=2$, $p\in\{2,3,5,7\}$ and $G=\mathbb{Z}_p\times\mathbb{Z}_p$;
\item[{\rm b)}] $\sigma_1(G)=2+\frac{11}{|G|}$\, if and only if $n=3$, $p=2$ and $G=\mathbb{Z}_2\times\mathbb{Z}_4$.
\end{itemize}
\end{lemma}

\begin{proof}
Since $G$ contains at least $p+1$ subgroups of order $p^{n-1}$, one obtains
\begin{equation}
2+\frac{11}{p^n}\geq\sigma_1(G)\geq\displaystyle\frac{1+p+\cdots+(p+1)p^{n-1}\!+p^n}{p^n}=2+\displaystyle\frac{1+p+\cdots+p^{n-1}}{p^n}\,,\nonumber
\end{equation}that is $1+p+\cdots+p^{n-1}\leq 11$. Thus either $n=3$ and $p=2$, or $n=2$ and $p\in\{2,3,5,7\}$. The conclusion follows by checking the corresponding groups.\qedhere
\end{proof}

Our second main result is stated as follows.

\begin{theorem}\label{th:C1}
    Let $G$ be a finite non-cyclic group. If $\sigma_1(G)=2+\frac{11}{|G|}$\,, then either $G=\mathbb{Z}_2\times\mathbb{Z}_4$ or $G=A_4$.
\end{theorem}

\begin{proof}
We can easily see that $G=\mathbb{Z}_2\times\mathbb{Z}_4$ is the unique group $G$ of order $\leq 11$ satisfying $\sigma_1(G)=2+\frac{11}{|G|}$\,. Assume that $|G|\geq 12$. Then $\sigma_1(G)<3$ and therefore, by Lemma 2.1, either $G$ is nilpotent or $G=(P\rtimes K_2)\times K_1$ with $P$, $K_1$ and $K_2$ as in Theorem 1.2, c).

In the first case, let $G=G_1\times\cdots\times G_k$ be the decomposition of $G$ as a direct product of Sylow $p$-subgroups. For $k=1$, we have no solution $G$ of $\sigma_1(G)=2+\frac{11}{|G|}$ by Lemma 2.3, b). For $k\geq 2$, we get $\sigma_1(G_i)<2+\frac{11}{|G_i|}$\,, $\forall\, i=1,...,k$. Then Lemma 2.3, a), shows that every $G_i$ is of one of the following types: $Q_8$, $\mathbb{Z}_{p_i}\times\mathbb{Z}_{p_i}$ or $\mathbb{Z}_{p_i^{n_i}}$ for some prime $p_i$ and some positive integer $n_i$. Since $$\sigma_1(Q_8)=\frac{23}{8}>2,\, \sigma_1(\mathbb{Z}_{p_i}\times\mathbb{Z}_{p_i})=2+\frac{p_i+1}{p_i^2}>2, \forall\, i=1,...,k$$and $$\sigma_1(G)=\sigma_1(G_1)\cdots\sigma_1(G_k),$$we infer that at most one of $G_i$'s is non-cyclic. Suppose that $G_1=Q_8$ and $G_i=\mathbb{Z}_{p_i^{n_i}}$, for all $i=2,...,k$. Then our condition becomes
$$\displaystyle\frac{23}{8}\,\frac{\sigma(p_2^{n_2})}{p_2^{n_2}}\cdots\frac{\sigma(p_k^{n_k})}{p_k^{n_k}}=2+\frac{11}{8p_2^{n_2}\cdots p_k^{n_k}}\,,$$i.e.
$$23\,\sigma(p_2^{n_2})\cdots\sigma(p_k^{n_k})=16p_2^{n_2}\cdots p_k^{n_k}+11\,,$$a contradiction (left side $>$ right side). Similarly, if we suppose $G_1=\mathbb{Z}_{p_1}\times\mathbb{Z}_{p_1}$ and $G_i=\mathbb{Z}_{p_i^{n_i}}$, for all $i=2,...,k$, then we obtain\newpage  $$\displaystyle\frac{2p_1^2+p_1+1}{p_1^2}\,\frac{\sigma(p_2^{n_2})}{p_2^{n_2}}\cdots\frac{\sigma(p_k^{n_k})}{p_k^{n_k}}=2+\frac{11}{p_1^2p_2^{n_2}\cdots p_k^{n_k}}\,,$$i.e.
$$(2p_1^2+p_1+1)\,\sigma(p_2^{n_2})\cdots\sigma(p_k^{n_k})=2p_1^2p_2^{n_2}\cdots p_k^{n_k}+11\,,$$a contradiction (left side $>$ right side). Thus every $G_i$ is cyclic, and so is $G$, contradicting our hypothesis.

In the second case, a similar reasoning shows that we have no solution $G=(P\rtimes K_2)\times K_1$ of $\sigma_1(G)=2+\frac{11}{|G|}$ with $K_1\neq 1$. Consequently, we may assume $G=P\rtimes K_2$. Let $|P|=p^r$ and $|K_2|=n$. Note that $p\nmid n$ and $n\mid |{\rm Aut}G|$. As in the proof of Theorem 2.2, we get $$2+\frac{11}{|G|}=\sigma_1(G)\geq 2+\displaystyle\frac{\displaystyle\sum_{H\leq P}|H|}{|G|}$$and thus $$\displaystyle\sum_{H\leq P}|H|\leq 11.$$It follows that either $r=2$, $p=2$ and $P\in\{\mathbb{Z}_4, \mathbb{Z}_2\times\mathbb{Z}_2\}$, or $r=1$, $p\in\{2,3,5,7\}$ and $P=\mathbb{Z}_p$.

If $P=\mathbb{Z}_4$, then $n\mid 2$, a contradiction. If $P=\mathbb{Z}_2\times\mathbb{Z}_2$, then $n\mid 6$ and so $n=3$; thus $G=A_4$, which is a solution of $\sigma_1(G)=2+\frac{11}{|G|}$\,. If $P=\mathbb{Z}_p$ with $p\in\{2,3,5,7\}$, then $n$ divides $1$, $2$, $4$ or $6$, respectively. We infer that the unique possibilities are $$(p,n)\in\{(3,2),(5,2),(5,4),(7,2),(7,3),(7,6)\}.$$By checking all these semidirect products $G=\mathbb{Z}_p\rtimes\mathbb{Z}_n$ we find no other solution of $\sigma_1(G)=2+\frac{11}{|G|}$\,. Hence in this case the unique solution is $G=A_4$, as desired.\qedhere
\end{proof}

Obviously, Theorem 2.4 leads to the following characterizations of $\mathbb{Z}_2\times\mathbb{Z}_4$ and $A_4$.

\begin{corollary}
$\mathbb{Z}_2\times\mathbb{Z}_4$ is the unique finite non-cyclic nilpotent group $G$ satisfying $\sigma_1(G)=2+\frac{11}{|G|}$\,, while $A_4$ is the unique finite non-nilpotent group $G$ satisfying $\sigma_1(G)=2+\frac{11}{|G|}$\,.
\end{corollary}

Finally, we remark that $\sigma_1(G)<\frac{31}{12}=\sigma_1(A_4)$ does not imply the supersolvability of $G$, a counterexample being $SmallGroup(56,11)$ (also known as the Frobenius group $F_8$). In fact, we are able to prove that there is no constant $c\in(2,\infty)$ such that $\sigma_1(G)<c$ implies the supersolvability of $G$.

\begin{theorem}\label{th:C1}
    There are sequences of finite non-supersolvable groups $(G_n)_{n\in\mathbb{N}}$ such that $\sigma_1(G_n)$ approaches $2$ monotonically from above as $n$ tends to in\-fi\-nity.
\end{theorem}

\begin{proof}
Let $(q_n)_{n\in\mathbb{N}}$ be the sequence of odd prime numbers. By Dirichlet's theorem we infer that for every $n\in\mathbb{N}$ there is a prime $p_n$ such that $q_n\mid p_n-1$. Let $G_n$ be the unique non-abelian group of order $p_n^2q_n$. Clearly, $G_n$ is not supersolvable. It contains one subgroup of order $1$, $p_n+1$ subgroups of order $p_n$, one subgroup of order $p_n^2$, $p_n^2$ subgroups of order $q_n$, and one subgroup of order $p_n^2q_n$. Then
\begin{equation}
\sigma_1(G_n)=\frac{1+(p_n+1)p_n+p_n^2+p_n^2q_n+p_n^2q_n}{p_n^2q_n}=2+\frac{2+\displaystyle\frac{1}{p_n}+\displaystyle\frac{1}{p_n^2}}{q_n}\,,\nonumber
\end{equation}which approaches $2$ monotonically from above as $n$ tends to infinity, as desired.\qedhere
\end{proof}

We end this paper by formulating a natural problem concerning the above study.
\bigskip

\noindent{\bf Open problem.} Find all finite cyclic groups $G$ satisfying $\sigma_1(G)=2+\frac{11}{|G|}$\,.
\bigskip

Note that this is equivalent to finding all positive integers $n=p_1^{n_1}\cdots p_k^{n_k}$ such that
$$\left(1+p_1+\cdots+p_1^{n_1}\right)\cdots\left(1+p_k+\cdots+p_k^{n_k}\right)=2p_1^{n_1}\cdots p_k^{n_k}+11.$$

\vspace*{5ex}\small

\hfill
\begin{minipage}[t]{5cm}
Marius T\u arn\u auceanu \\
Faculty of  Mathematics \\
``Al.I. Cuza'' University \\
Ia\c si, Romania \\
e-mail: {\tt tarnauc@uaic.ro}
\end{minipage}

\end{document}